\documentclass[12pt,letterpaper]{article}

\usepackage[centertags]{amsmath}
\usepackage{amssymb}
\usepackage{amsthm}

\newcommand{\be}{\begin{enumerate}}
\newcommand{\ee}{\end{enumerate}}

\newcommand{\bc}{\begin{center}}
\newcommand{\ec}{\end{center}}

\newcommand{\beqn}{\begin{eqnarray}}
\newcommand{\eeqn}{\end{eqnarray}}
\newcommand{\bequ}{\begin{equation}}
\newcommand{\eequ}{\end{equation}}
\newcommand{\tien}{\rightarrow}
\newcommand{\nn}{\nonumber}
\newtheorem{teo}{Theorem}[section]
\newtheorem{cor}[teo]{Corollary}
\newtheorem{lem}[teo]{Lemma}
\newtheorem{defi}{Definition}[section]
\newtheorem{prop}[teo]{Proposition}
\newtheorem{rem}[teo]{Remark}
\newcommand{\ldos}{L^{2}(\mathbb{R}^n)}

\newcommand{\linf}{L^{\infty}(\mathbb{R}^n)}
\newcommand{\lp}{L^{p}(\mathbb{R}^n)}
\newcommand{\lph}{L_{p}}

\newcommand{\Sch}{\mathcal{S}(\Rn)}
\newcommand{\Schp}{\mathcal{S}'(\Rn)}

\newcommand{\Fou}{\mathcal{F}}
\newcommand{\fou}{\widehat}
\newcommand{\fourier}{\mbox{Fourier}}

\newcommand{\gbeta}{\mathcal{G}^{\beta}}

\newcommand{\hrp}{H^{r,p}(\mathbb{R}^n)}
\newcommand{\Hsp}{\mathcal{H}^{s,p}(a)}
\newcommand{\Hspr}{\mathcal{H}^{s,p}_r(a)}

\newcommand{\Rn}{\mathbb{R}^n}
\newcommand{\tiend}{\rightarrow}

\title{Nonlinear pseudo-differential equations defined by elliptic symbols on ${\lp}$ and the fractional Laplacian}
\author{Mauricio Bravo, Humberto Prado\thanks{E-mail: humberto.prado@usach.cl} \; and
        Enrique G. Reyes\thanks{E-mail: e\_g\_reyes@yahoo.ca ; enrique.reyes@usach.cl} \\ \smallskip \\
Departamento de Matem\'{a}tica y Ciencia de la Computaci\'{o}n,\\
        Universidad de Santiago de Chile.\\
        Casilla 307 Correo 2, Santiago, Chile }
\date{\today}

\begin{document}

\maketitle

\begin{abstract}
We develop an $L^p(\mathbb{R}^n)$-functional calculus appropriated for interpreting ``non-classical symbols" of the
form $a(-\Delta)$, and for proving existence of solutions to nonlinear
pseudo-differential equations of the form $[1 + a(-\Delta)]^{s/2} (u) = V( \cdot , u)$. We
use the theory of Fourier multipliers for constructing suitable domains sitting inside $L^p(\mathbb{R}^n)$
on which the formal operator appearing in the above equation can be
rigorously defined, and we prove existence of solutions belonging to these domains.
We include applications of the theory to equations of physical interest involving the fractional
Laplace operator such as (generalizations of) the (focusing) Allen-Cahn, Benjamin-Ono and nonlinear Schr\"odinger equations.
\end{abstract}

\section{Introduction}
The aim of this work is to study scalar nonlinear equations of the form
\begin{equation} \label{abs}
[1 + a(-\Delta)]^{s/2} (u) = V( \cdot , u) \; ,
\end{equation}
in which $\Delta$ stands for the Laplace operator on $\mathbb R^n$ and the function $a(t)$, $t\geq 0$, is a
``non-classical symbol" of elliptic type,
see Definition 2.1 below. Two of the present authors, in collaboration with P. G\'orka, have developed a mathematical
framework which allows us to understand equations of the form (\ref{abs}) in the $L^2(\mathbb{R}^n)$ context, see
\cite{G-P-R,G-P-R5,G-P-R3}. In this article we consider these equations in the completely general context of Lebesgue
spaces. We show that, as is the case for classical elliptic equations, we can prove existence of solutions to
(\ref{abs}) with integrability and differentiability properties that go beyond the ones we would expect for
$L^p(\mathbb{R}^n)$ functions. We recall that equations depending on symbols $a(-\Delta)$ belong to a class of
pseudo-differential
equations which have been treated in a formal manner as ``equations in infinitely many derivatives" in the Physics
literature, see for instance \cite{Bar,Bar3,Bar4,Cal2,Cal,Vla}. These also include non-local equations, as observed
recently in \cite{CPR}, and equations of physical interest depending on fractional Laplace operators, see for instance
\cite{ABFS,va,va2,BV,CS1,CS2,FL,FLS,FJL,L,Sir}, the older paper \cite{Du}, and also \cite{KSM} for a {\em numerical} study
of the fractional nonlinear Schr\"odinger equation.

It is very important to stress the fact that the approach we followed in
\cite{G-P-R,G-P-R5,G-P-R3} does not apply in an obvious way to the  $L^p(\mathbb{R}^n)$ framework.
To give a basic example, we note that some of the arguments used in \cite{G-P-R,G-P-R5,G-P-R3} break
down because of the non-existence of a Plancherel theorem in $L^p(\mathbb{R}^n)$ (or, in other words
because the Fourier transform is not a unitary operator on $L^p(\mathbb{R}^n)\,$). Nevertheless, in
this paper we show that it is possible to develop alternative techniques to deal with equations such
as (\ref{abs}) in general Lebesgue spaces.

Our main tool is the use of Fourier multipliers after \cite{Ar,Gui,Ste}. These multipliers allow us to introduce a scale
of spaces $\mathcal{H}^{s,p}(a)$, sitting inside $L^p(\mathbb{R}^n)$, which are determined by the symbol $a$. Partially
motivated by Lions' classical paper \cite{Li}, we look at some embedding properties of the spaces $\mathcal{H}^{s,p}(a)$,
and we also consider their relation with fractional Sobolev
spaces. In particular, in the case of $L^p(\mathbb{R}^n)$ spaces of radial functions, we obtain a
``radial scale" of spaces $\mathcal{H}^{s,p}_r(a)$, see Section 4 below. Lions' results in \cite{Li} imply that
$\mathcal{H}^{s,p}_r(a)$ is compactly embedded into an appropriate radial space $L^q_r(\mathbb{R}^n)$.

We organize our work as follows. In Section 2 we introduce the class
$\mathcal{G}^\beta_s$ of symbols we consider thereafter, see Definition
2.1. An important example of allowable symbol is the fractional Laplacian $(-\Delta + m^2)^{\gamma/2}$, $0 < \gamma < 1$ and $m \neq 0$,
as we show in Lemma \ref{xxx}. In Section 3 we introduce Fourier multipliers and we develop a functional
calculus in $L^p(\mathbb{R}^n)$ for symbols $a$ in $\mathcal{G}^\beta_s$.
This functional calculus takes into account the identification of the operator $a(-\Delta)$ with the symbol
$a(|\xi|^2)$, motivated by the well-known fact that in the $L^2(\mathbb{R}^n)$ setting this correspondence is rigorously
achieved with the aid of the Fourier transform, see for instance \cite{G-P-R} (and, for some special symbols, by exploiting the
existence of a dense set of analytic vectors for $\Delta$ in $L^2 (\mathbb{R}^n)$, see \cite{G-P-R3}).
We then define a bona fide operator with domain
$\mathcal{H}^{s,p}(a) \subseteq  L^p(\mathbb{R}^n)$  which corresponds to the formal operator appearing on the left
hand side of Equation (\ref{abs}). The embedding properties mentioned in the previous paragraph are also considered in
this section. In Section 4 we consider Equation (\ref{abs}) and we prove two existence theorems: the first in a
``localized" setting, and the second one in the setting of functions with radial symmetry,
see Theorems \ref{Teo ec no lineal} and \ref{radi} respectively. Let us paraphrase Theorem \ref{radi}:

\noindent
We assume that $a$ is an allowable symbol, that the nonlinear function $V(x,y)$ is spherically symmetric with respect to
$x \in \Rn$, and that the natural growth conditions
$$
| V(x,y) | \leq C (|h(x)| + |y|^\alpha) \quad \mbox{ and } \quad
\left| \frac{\partial}{\partial y} V(x,y) \right| \leq C ( |g(x)| + |y|^{\alpha -1})
$$
hold for some $\alpha > 1$, $C > 0$, and functions $h \in \lp$ and $g\in L^{\frac{\alpha p}{\alpha-1}}(\mathbb{R}^n)$.{\em Then},
Equation (\ref{abs}) has a solution $u \in \mathcal{H}^{s,p}_r(a)$. The aforementioned compact embedding theorems appearing in \cite{Li} are
crucially used in the proof of this theorem. Moreover, they allow us to prove that our solution not only belong to
the radial space $\mathcal{H}^{s,p}_r(a) \subseteq  L^p_r(\mathbb{R}^n)$, but that actually it belongs to $L^{\alpha p}_r(\mathbb{R}^n)$!
Theorem \ref{radi} also includes a bound on the $L^{\alpha p}_r(\mathbb{R}^n)$-norm of the solution $u$.

\noindent
Finally, in Section 5 we apply our radial existence result to prove
existence theorems for equations of physical interest involving the fractional Laplace operator
$a_m(\Delta) = (-\Delta + m^2)^{\gamma/2}$, $0 < \gamma < 1$, $m \neq 0$, and we also discuss briefly the case
of equations depending on $a_0(\Delta) = (-\Delta)^{\gamma/2}$.
Particular examples considered in this section are
(generalizations of) a (perturbed, focusing) fractional Allen-Cahn equation, the Benjamin-Ono equation, a (perturbed) Peierls-Nabarro equation,
and a fractional non-linear Schr\"odinger equation.

\section{Preliminaries}

We begin by defining a class of functions which will be crucial for our work.

\begin{defi} \label{2.1}
Given $s, \beta>0$ and $n \in \mathbb{N}$ such that $\beta\,s \geq 4 n$, we say that a scalar function
$a$ on $\mathbb{R}$ belongs to the class $\gbeta_s(\mathbb R^n)$ if it satisfies the following three conditions:
\be
\item[$G_1$.] The function $a$ belongs to $C^n([0,\infty))$
and the function $t\mapsto a(t^2)$, $t\in\mathbb R$, is non-negative.
\item[$G_2$.] There exist positive constants $R,\ M$ such that
\begin{equation}\label{elipticidad}
M(1+|x|^2)^{\frac{\beta}{2}}\leq a(|x|^2),\quad\mbox{ for all}\quad |x|>R\; , \quad\quad x\in\mathbb R^n \; .
\end{equation}
\item[$G_3$.] For each $k > 0$ there exists a positive constant $N=N(k,s)$ and $\rho \geq 0$ such that, whenever $|x| > \rho$,
the estimate
\begin{equation}\label{crec derivadas}
\left| \,\left.\dfrac{d^{k}}{dt^{k}}a(t)\,\right|_{t=|x|^2} \right| \leq
N(1+|x|^2)^{k(\frac{\beta s}{4n}-1)+\frac{\beta}{2}}
\end{equation}
holds.
\end{enumerate}
\end{defi}

In the above definition we can assume  that $R = \rho$ without loss of generality.
An interesting example of a function $a$ belonging to $\gbeta_s(\mathbb R^n)$ is given in the following lemma,
inspired by the equations studied in \cite{va,FJL,L}:

\begin{lem} \label{xxx}
We take $0 < \gamma < 1$ and we assume that $s > 4n/\gamma$. Then, the function $a(t) = (|t|+m^2)^{\gamma/2}$, $m \neq 0$,
belongs to $\mathcal{G}^\gamma_s(\mathbb R^n)$.
\end{lem}
\begin{proof}
Condition $G_1$ is obvious. In order to show that $a$ satisfies $G_2$,
we note that for $|x|\geq 1$  we have
$$
(1+|x|^2)^{\gamma/2} \leq (|x|^2+|x|^2)^{\gamma/2} \leq 2^{\gamma/2}\,(m^2 + |x|^2)^{\gamma/2} \; .
$$
Finally for $G_3$, we note that for each $k > 0$ there exists a constant number $C(k)$ such that $a^{(k)}(t)=C(k)\, (t+m^2)^{(\gamma-2k)/2}$,
$t \geq 0$. We have, for $|x| > 1$,
$$
|a^{(k)}(|x|^2)|=C(k)\left|\, |x|^2 + m^2 \right|^{(\gamma-2 k)/2}  < C(k)
< C(k) (1+|x|^2)^{k(\frac{\gamma s}{4n}-1)+\frac{\gamma}{2}}\; ,
$$
since $\gamma-2 k < 0$ and we are assuming that $s > 4n/\gamma$.
\end{proof}

\smallskip

We note that if we relax the condition $\beta s /4 n > 1$ in the definition of $\gbeta_s(\mathbb R^n)$, we can include
the H\"ormander class $S^m$, see \cite[Definition 18.1.1]{Hor3}, into $\gbeta_s(\mathbb R^n)$ with $\beta = 2m$ and $s=0$.
Further examples of symbols in our class are given by functions $a(t)$ satisfying $G_1,G_2$ and the estimates
$$
\left| \frac{d^k}{dt^k} a(t) \right| \leq N_k |t|^{k R + \beta/2}
$$
for constant numbers $N_k ,R > 0$, as it is easy to see.

\smallskip

The following proposition is an easy consequence of the
definitions. It allows us to order the classes $\gbeta_s(\mathbb R^n)$.

\begin{prop}\label{contencion gbetas}
Let us fix a number $\beta>0$. If $s_1$ and $s_2$ are real numbers such
that $0\leq s_1 < s_2$ and $\beta s_1 \geq 4 n$, then $\gbeta_{s_1}(\mathbb R^n) \subseteq \gbeta_{s_2}(\mathbb R^n)\, .$
\end{prop}
\begin{proof}
Let us consider $a\in\gbeta_{s_1}(\mathbb R^n)$. Conditions $G_1$ and $G_2$
depend only on $\beta$, so we only prove
that $G_3$ holds if we use $s_2$ instead of $s$. The result follows from the obvious facts that
$s_1<s_2$ implies $\dfrac{\beta s_1}{4n} < \dfrac{\beta s_2}{4n}$ and that the mapping
$t \mapsto (1+|x|^2)^t$, $t > 0$, is increasing.
\end{proof}

Now we take $\mu>0$. We denote by $m_{a,\mu}$ the function
\begin{equation}\label{funcion m}
m_{a,\mu}(x)=\dfrac{1}{(1+a(|x|^2))^{\mu/2}}\; , \quad\quad\quad x\in \mathbb R^n\; .
\end{equation}

We will prove in the next section that if $a$ belongs to an appropriate class $\gbeta_s(\mathbb{R}^n)$ and $\mu$ is chosen
accordingly, the function $m_{a,\mu}$ is an $L^p(\mathbb{R}^n)$ Fourier multiplier (see Definition 3.1 below).
As a preparation, we consider the derivatives of $m_{a,\mu}\,.$

We let $I = ( i_1, i_2, \cdots , i_k )$ be a multi-index
of {\em size} $k$ with $1 \leq i_k \leq n$; we will always assume that multi-indexes are ordered, this is, $i_r \leq i_s$
if $r \leq s$. We say that $J = (j_1, \cdots, j_r)$ is contained in $I$ if $r \leq k$ and for each $1 \leq p \leq r$ we
have $j_p = i_q$ for some $1 \leq q \leq k$.  We define collections of subsets of $I$ as follows:

\begin{itemize}
\item[~] $C_r$ is the set of all collections $\{I_1, \cdots , I_r \}$ of subsets of $I$ which are pairwise disjoint
and satisfy $I_1 \cup \cdots \cup I_r = I$.
\end{itemize}

\noindent We also set $\partial_I ( u ) = \partial_{x_{i_1}} \cdots \partial_{x_{i_k}} (u)$. A straightforward computation
yields the formula
\begin{eqnarray}
\partial_I (m_{a,\mu}) & = & A_I\, \partial_I(a) \,m_{a,\mu+2} \, +
    \sum_{\{I_1,I_2 \} \in C_2} A_{I_1,I_2}\, \partial_{I_1}(a) \partial_{I_2}(a)\,m_{a,\mu+4} \, + \nonumber\\
&  &\sum_{\{I_1,I_2,I_3 \} \in C_3} A_{I_1,I_2,I_3}\, \partial_{I_1}(a) \partial_{I_2}(a) \partial_{I_3}(a) \,m_{a,\mu+6} + \cdots + \nonumber \\
&  &\sum_{\{I_1,I_2, \cdots , I_k \} \in C_k} A_{I_1,I_2, \cdots , I_k}\, \partial_{I_1}(a) \partial_{I_2}(a) \cdots \partial_{I_k} (a) \,m_{a,\mu+2k} \; ,
\label{derivada de m}
\end{eqnarray}
\noindent in which $A_I, A_{I_1,I_2}, \cdots$ are constant numbers. We will need this formula in the case all components
$i_p$ of $I = ( i_1, i_2, \cdots , i_k )$ are different. Thus, we complement (\ref{derivada de m})
with the obvious observation
\begin{equation} \label{derivada de m1}
\partial_{x_{i_1}} \cdots \partial_{x_{i_k}}(a(|x|^2)) = 2^k x_{i_1}\cdots x_{i_k}a^{(k)}(|x|^2) \quad \quad
\mbox{ all $i_p$ different} \; .
\end{equation}

\section{Fourier multipliers and functional calculus}   \label{section Fourier multip}
\subsection{Fourier multipliers}

Hereafter we write either $\Fou(f)$ or $\widehat{f}$ for the Fourier transform of $f$.

\begin{defi}\label{def multip}
Let $m$ be a bounded measurable function on $\Rn$. We define a
linear transformation $T_m$ with domain $\ldos\cap\lp$ as follows:
\begin{equation}\label{mult de fourier}
T_m f = \Fou^{-1}\left( m(\xi)\fou{f} \, \right) ,\quad f\in\ldos\cap\lp \; .
\end{equation}
We say that $m$ is a Fourier multiplier for $\lp$, $1 < p <
\infty$, if whenever $f\in\ldos\cap\lp$ then $T_m f$ belongs to
$\lp$ and $T_m$ is bounded, that is,
\begin{equation}
\|T_m(f)\|_{\lp}\leq A\|f\|_{\lp}, \quad f\in\ldos\cap\lp\quad
\end{equation}
for a real constant $A$ independent of $f$.
\end{defi}

This definition is explained in \cite[p. 94]{Ste} and \cite[p. 489]{Ar}. Its
importance for us is due to Proposition \ref{mfs multiplicador} below, in which the relevance
of our conditions $G_1$--$G_3$ become clear. The proof of this proposition is based on the
following lemma, see \cite[Theorem 2]{Gui}.

\begin{lem} \label{gui}
Let $m \in  C^n(\mathbb{R}^n \setminus \{0\})$ be such that for every multi-index $\alpha \in \mathbb{N}^n_0$ with $\alpha \leq (1, \dots , 1)$,
and every $x \in \mathbb{R}^n \setminus \{0\}$,
\begin{equation}
|x_\alpha\, D^\alpha m(x)| \leq  C
\end{equation}
for some $C > 0$. Then $m$ is a Fourier multiplier for $\lp$ for every $1 < p < \infty$.
\end{lem}

\noindent We have:

\begin{prop}\label{mfs multiplicador}
Let $\beta , s > 0$ such that $\beta \, s \geq 4\, n$, and let $a\in\gbeta_{s}(\mathbb{R}^n)$. If $1<p<\infty$,
then the function  $m_{a,\mu}$ defined in $(\ref{funcion m})$ is a Fourier multiplier for $\lp$ whenever $\mu \geq s$.
\end{prop}
\begin{proof}
We use the multi-index convention appearing in the computations leading to (\ref{derivada de m}).
Let us show that for every multi-index $I=(i_1,i_2,\ldots,i_k)$ with $i_1 < i_2 < i_3 < \cdots < i_k$,
there exists a constant $C$ such that for all $x\in\Rn$
\begin{equation} \label{aux1}
|x_{i_1} x_{i_2} \cdots x_{i_k}\,\partial_{x_{i_1}}\, \cdots \partial_{x_{i_k}}\, m_{a,\mu}(x)| \leq C
\end{equation}
whenever $\mu \geq s$. It follows from Lemma \ref{gui}, see also Theorem IV.6' of \cite{Ste}, that this is
enough for proving the proposition.

First, we use that for every $1\leq i\leq n$ we have that $|x_i| \leq (1+|x|^2)^{1/2}\,$. Thus,
\begin{equation} \label{aux111}
|x_{i_1} x_{i_2} \cdots x_{i_k}\,\partial_{x_{i_1}}\, \cdots \partial_{x_{i_k}}\, m_{a,\mu}(x)| \leq
(1+|x|^2)^{k/2}\,|\partial_{x_{i_1}}\, \cdots \partial_{x_{i_k}}\, m_{a,\mu}(x)| \; .
\end{equation}
Now we set $I =(i_1, \cdots , i_k)$ and we estimate $|\partial _I m_{a,\mu}(x)|$ using formula (\ref{derivada de m}).
We bound each one of the summands appearing therein. for each $r=1, \cdots, k$  we consider

\noindent
$
S_r = m_{a,\mu+2r}(x) \sum_{\{I_1,\cdots,I_r \} \in C_r} A_{I_1, \cdots ,I_r}\,\partial_{I_1}(a)\cdots \partial_{I_r}(a)\; .
$
We have
\begin{equation} \label{aux11}
|S_r| \leq | m_{a,\mu+2r}(x) | \sum_{\{I_1,\cdots,I_r \} \in C_r} |\partial_{I_1}(a)| \cdots |\partial_{I_r}(a)| \; ,
\end{equation}
in which we have omitted writing the constant number $A_{I_1, \cdots ,I_r}$, as it is actually superfluous for proving
(\ref{aux1}). We begin by estimating $| m_{a,\mu+2r}(x) |$ for $|x|$ large enough:
\begin{equation} \label{aux2}
|m_{a,\mu+2r}(x)|=\left|\frac{1}{(1+a(|x|^2))^{\frac{\mu+2r}{2}}}\right|\leq\frac{1}{(1+|x|^2)^{\frac{\beta(\mu+2r)}{4}}}\; ,
\end{equation}
in which we have used condition $G_2$, see (\ref{elipticidad}). Now we estimate each factor $|\partial_{I_q} (a(|x|^2))|$,
$1 \leq p \leq r$ appearing in (\ref{aux11}). Let us assume that $I_q =(j_1, \cdots,j_{p_q})$. Then, (\ref{derivada de m1})
implies
\begin{eqnarray}
|\partial_{I_q} (a(|x|^2))| & = & 2^{p_q} \left| x_{j_1}\cdots x_{j_{p_q}}a^{(p_q)}(|x^2|) \right| \nonumber \\
& \leq & 2^{p_q} (1+|x|^2)^{\frac{p_q}{2}} \left| a^{(p_q)}(|x^2|) \right| \nonumber \\
& \leq & (1+|x|^2)^{\frac{p_q}{2}}\, (1+|x|^2)^{p_q\left(\frac{\beta s}{4 n} - 1\right) + \frac{\beta}{2}}
\end{eqnarray}
for $|x|$ large enough, in which we have used condition $G_3$, see (\ref{crec derivadas}), and we have omitted constant
numbers which are irrelevant for our present goal. Now we go back to (\ref{aux11}). We obtain
\begin{eqnarray}
|S_r| & \leq & (1+|x|^2)^{-\frac{\beta(\mu+2r)}{4}}  \sum_{\{I_1,\cdots,I_r \} \in C_r} (1+|x|^2)^{\frac{k}{2}}
 (1+|x|^2)^{k\left(\frac{\beta s}{4 n} - 1\right) + r \frac{\beta}{2}} \nonumber \\
 & \leq & (1+|x|^2)^{\frac{k}{2}-\frac{\beta(\mu+2r)}{4}+k\left(\frac{\beta s}{4 n} - 1\right)}
   \sum_{\{I_1,\cdots,I_r \} \in C_r}(1+|x|^2)^{r \frac{\beta}{2}}
 \nonumber \\
 & \leq & (1+|x|^2)^{\frac{k}{2}-\frac{\beta(\mu+2r)}{4}+k\left(\frac{\beta s}{4 n} - 1\right)+ r \frac{\beta}{2}}
 \nonumber \\
 & = & (1+|x|^2)^{-\frac{k}{2}-\frac{\beta \mu}{4}+k \frac{\beta s}{4 n}} \; ,\label{aux222}
\end{eqnarray}
in which, once more, we have absorbed irrelevant constant numbers. We can go back to (\ref{aux111}). Multiplying
(\ref{aux222}) by $(1+|x|^2)^{\frac{k}{2}}$ we obtain that each summand $(1+|x|^2)^{\frac{k}{2}}\,S_r$ of (\ref{aux111})
is bounded by
$$
(1+|x|^2)^{-\frac{\beta \mu}{4}+k \frac{\beta s}{4 n}} = (1+|x|^2)^{\frac{\beta}{4}\left(-\mu + \frac{k}{n}\,s\right)} \; .
$$
Since $k \leq n$, the exponent $\frac{\beta}{4}\left(-\mu + \frac{k}{n}\,s\right)$ is at most zero whenever $\mu \geq s$.
Thus, (\ref{aux111}) implies that the uniform bound (\ref{aux1}) holds for $|x|$ large enough. Condition $G_1$ and the fact that
$a \in C^n( [0,\infty))$ allow us to conclude that (\ref{aux1}) holds for all $x$. It follows from Lemma \ref{gui} that the function
$m_{a,\mu}$ is a Fourier multiplier for $\lp$ for every $p\in(1,\infty)$ and for all $\mu \geq s > 0$.
\end{proof}

Now we take  $\beta , s > 0$, $\beta \, s \geq 4\, n$, and we consider a symbol $a \in \gbeta_{s}(\mathbb{R}^n)$, the
Fourier multiplier $m_{a,s}\,$, and the induced linear transformation $T_{m_{a,s}}$ introduced in Definition 3.1. An
important characteristic of $T_{m_{a,s}}$ is the following invariance property:

\begin{prop}\label{Ts op invariante}
Let us set $T_s = T_{m_{a,s}}$ on $\lp$, $1 < p < \infty$, that is,
\begin{equation}\label{operador Ts}
T_s(u) = \Fou^{-1}\left(\dfrac{1}{(1+a(|\xi|^2))^{s/2}}\,\fou{u}\right).
\end{equation}
Then, $T_s$ is a bounded translation-invariant operator from $L^2(\mathbb{R}^n) \cap \lp$ to $\lp$.
\end{prop}
\begin{proof}
Since $m_{a,s}$ is a Fourier multiplier, we know that $T_s$ is a bounded
linear operator from $L^2(\mathbb{R}^n) \cap \lp$ to $\lp$. We write
$$T_s(u) = \Fou^{-1}\left(m_{a,s}(\xi)\,\fou{u}\,\right).$$
Then, invariance follows from well-known properties of Fourier transform, and
the fact that $m_{a,s}$ is rotationally invariant.
\end{proof}

We end this subsection by observing that if $p < \infty$, Definition \ref{def multip} implies that $T_s$ has a
unique bounded extension to $\lp$. Following standard usage, we keep writing $T_s$ for this extension. This observation
appears in Stein's book \cite[p. 94]{Ste}.

\subsection{Functional calculus for $\gbeta_s(\mathbb R^n)$ symbols}

First of all we observe that (under suitable conditions) operators $T_s$ act as convolution operators with
a kernel $K$ which we now introduce.

\begin{lem}\label{K en L uno}
Let $s , \beta$ such that $\beta s>4n$, and let $a\in\gbeta_s(\mathbb R^n)$. Let us consider
\begin{equation}\label{definicion K}
K(x)=\int_{\Rn}\dfrac{e^{i\xi x}}{(1+a(|\xi|^2))^{s/2}}d\xi\; .
\end{equation}
Then, $K\in\ldos$.
\end{lem}
\begin{proof}
We note that the ellipticity condition $(G_2)$ yields the inequality
\begin{equation}
\dfrac{1}{(1+a(|\xi|^2))^{s/2}}\leq \dfrac{M}{(1+|\xi|^2)^{\frac{\beta\,s }{4}}}
\end{equation}
for $|x|>R$, and therefore $m_{a,s} \in \ldos$. It follows that $K = \mathcal{F}(m_{a,s})$ belongs to $\ldos$.
\end{proof}


We now define an operator $A$ on $\lp$ for each symbol $a\in \gbeta_s(\mathbb R^n)$. We
assume without further mention that each time we consider Fourier transform (or
inverse Fourier transform) of elements in $\lp$, we are identifying these elements with appropriate tempered distributions.

\begin{defi}\label{defi D(A)}
Let $a\in \gbeta_s(\mathbb R^n)$, $s , \beta>0$. We set
$$A(u):=\Fou^{-1}\left( \left(1+a(|\xi|^2) \right)^{s/2}\,\fou{u}\,\right)$$
on $\lp$, $1 < p < \infty$, with domain
$$D(A)=\{u\in\lp:\Fou^{-1}((1+a(|\xi|^2))^{s/2}\,\fou{u})\in\lp\}.$$
\end{defi}

\begin{rem} \label{te-ese}
We need to check that $D(A)$ is non-trivial. Indeed, let $a\in\gbeta_s(\mathbb R^n)$, $\beta>0$, $s\geq 0$.
We consider the bounded operator $T_s$ defined in Proposition $\ref{Ts op invariante}$. It is easy to see, using the fact
that Fourier transform is an isomorphism at the level of tempered distributions,
that $u \in \lp$ belongs to $D(A)$ if and only if
\begin{equation} \label{iso}
u = T_s (g)
\end{equation}
for some $g\in\lp$. This argument also shows that $T_s$ is invertible and $T_s^{-1} = A$.
\end{rem}

\smallskip

Now we show that $T_s$ acts as a convolution operator on an appropriate domain. We use the following characterization result on bounded invariant
translation operators on $\lp$ appearing in Theorem 1.2 of H\"ormander's paper \cite{Hor}: if $T$ is such an operator, there exists a unique
tempered distribution $\Theta \in \Schp$ such that $T g = \Theta \ast g$.

\begin{prop} \label{defi D(A)4}
Let $a\in\gbeta_s(\mathbb R^n)  \cap C^\infty(\mathbb{R})$, $\beta s > 4n$, and let us denote by
$\Sch$ the Schwartz space. If a function $u$ belongs to $D(A) \cap \Sch$ then
$u = T_s (g) = K \ast g\,$ for $g \in \Sch$, in which $K$ has been defined in $(\ref{definicion K})$.
\end{prop}
\begin{proof}
By the remark above, if $u$ belongs to $D(A) \cap \Sch$, we can write $u$
as in (\ref{iso}) for some $g\in\lp$. Equation (\ref{iso}) then implies
that $$(1+a(|\xi|^2))^{s/2}\, \widehat{u} = \widehat{g}\; .$$
Since $u$ belongs to $D(A) \cap \Sch$ and the function $\xi \mapsto a(|\xi|^2)$ is polynomially bounded,
the left hand side of the above equation belongs to Schwartz space, and therefore so does $g$.

\smallskip

Now, by Proposition \ref{mfs multiplicador}, the
function  $M(\xi)=\dfrac{1}{(1+a(|\xi|^2))^{s/2}}$ is a Fourier
multiplier for $\lp$ and therefore (see Proposition \ref{Ts op invariante} and
remark below it) the unique continuous extension of the map
\begin{equation} \label{iso1}
g\mapsto T_s(g)=\Fou^{-1}\left(\dfrac{1}{(1+a(|\xi|^2))^{s/2}}\fou{g}(\xi)\right)
\end{equation}
is a translation invariant bounded operator from $\lp$ into $\lp$. In particular, if $g\in \Sch$, we compute
$T_s(g)$ using exactly (\ref{iso1}).  By Theorem 1.2 in H\"ormander's paper \cite{Hor}, we conclude that there exists
a unique distribution $\Theta \in \Schp$ such that
\begin{equation}
T_s(g) = \Theta \ast g
\end{equation}
for all $g \in \Sch$. We claim that $\Theta$ is exactly $K$ as defined in
(\ref{definicion K}). Indeed,
$$\Fou(T_s(g)) = M \, \widehat{g} = \widehat{\widehat{M}} \, \widehat{g}
= \Fou ( \widehat{M} \ast g ) \; .$$
Thus, $T_s(g) = \widehat{M} \ast g$ for $g \in \Sch$. By uniqueness of the
distribution $\Theta$, we have $\Theta = \widehat{M} = K$.
\end{proof}

We are now ready to make the following crucial definition:

\begin{defi} \label{defi D(A)2}
Given $s , \beta>0$ and $a$ in the class $\gbeta_s(\mathbb R^n)$, we
define the space $\Hsp$, $1 < p < \infty$, as follows:
$$\Hsp=\left\{ u\in\lp:\Fou^{-1}(\,(1+a(|\cdot|^2))^{s/2}\Fou(u) \,)\in\lp \right\}\; .$$
\end{defi}

\smallskip

Thus, if $s , \beta>0$ and $a \in \gbeta_s(\mathbb R^n)$ then, $D(A)=\Hsp$.

\begin{rem}
The $\ldos$ case of Definition $\ref{defi D(A)2}$ was treated in {\rm \cite{G-P-R}}.
Interestingly, that paper was motivated by our previous work on the non-linear
equation $$\Delta \exp(-c \Delta) u = U( \cdot , u) \; ,$$ see
{\rm \cite{G-P-R3,G-P-R5}}. In the context of the
present paper it is natural to consider the function $a(t) = t \exp(c\, t)$ as the corresponding symbol associated to the
operator $-\Delta \exp(-c \Delta)$, and to study the equation $\Delta \exp(-c \Delta) u = -U( \cdot , u)$ instead of
the equation above. It is not difficult to realize that the function
\begin{equation}\label{funcion m exp}
m_{exp}(x)=\dfrac{1}{1+|x|^2 e^{c |x|^2}}
\end{equation}
defines a Fourier multiplier for $\lp$, even though the function $t \mapsto t e^{c t}$ does not belongs to
any of the spaces of symbols $\gbeta_s(\mathbb R^n)$. It follows that we can also use the
symbol $a(t) = t \exp(c\, t)$ in Definitions $\ref{defi D(A)}$ and $\ref{defi D(A)2}$, thereby yielding a framework for an $\lp$-version
of the theory developed in {\rm \cite{G-P-R3,G-P-R5}}. We will consider this elsewhere, since it is beyond the scope of this paper.
\end{rem}

\smallskip

We endow the space $\Hsp$ with the norm $\| \cdot \|_{\Hsp}$ defined in (\ref{ecc}) below:

\begin{teo} \label{bi}
Let $s ,\beta>0$ and $a \in \gbeta_s(\mathbb R^n)$. The space $\Hsp$, $1 < p < \infty$,
endowed with the norm
\begin{equation} \label{ecc}
\|u\|_{\Hsp}=\|\Fou^{-1}( \,(1+a(|\cdot|^2))^{s/2}\Fou (u) \,) \|_{\lp} \; .
\end{equation}
is a Banach space.
\end{teo}

\noindent The proof is standard. We note that the operator $T_s$ is continuous on $\lp$ and bijective, see Proposition
\ref{Ts op invariante} and Remark \ref{te-ese}. Its inverse is, explicitly,
 $T_s^{-1} : u \mapsto \Fou^{-1}( \,(1+a(|\cdot|^2))^{s/2}\Fou (u) \,)$. Thus, (\ref{ecc}) can be written as
 \begin{equation} \label{neweq}
 \| u \|_{\Hsp} = \| T_s^{-1}(u) \|_{\lp}\; ,
 \end{equation}
so that $T_s^{-1}$ is continuous on $\Hsp$ and $\Hsp$ is indeed a Banach space.

\smallskip

\smallskip

The Banach space $\Hsp$ embeds continuously into $\lp$. Indeed:
 \begin{teo}\label{embed on Lp}
Let $s , \beta>0$ and $a\in\gbeta_s(\mathbb R^n)$. Then the inclusion map
$\Hsp \hookrightarrow \lp$, $1 < p < \infty$, is continuous.
\end{teo}
\begin{proof}
Observe that for all $u\in\lp$
\begin{eqnarray}
\|u\|_{\lp}&=&\left\|\Fou^{-1}\left(\dfrac{1}{(1+a(|\xi|^2))^{s/2}} \,
(1+a(|\xi|^2))^{s/2}\Fou(u)\right)\right\|_{\lp}\nn\\
&=&\left\|T_s\left(\Fou^{-1}((1+a(|\xi|^2))^{s/2}\Fou(u))\right)\right\|_{\lp} \; . \label{|Ts Fu|}
\end{eqnarray}
Since $T_s$ is a bounded operator on $\lp$, there exists a positive constant $C$ such that
\begin{eqnarray*}
\left\|T_s\left(\Fou^{-1}((1+a(|\xi|^2))^{s/2}\Fou(u))\right)\right\|_{\lp}
& \leq & C\|\Fou^{-1}(((1+a(|\xi|^2))^{s/2}\Fou(u))\| \\
 & = & C\|u\|_{\Hsp} \; .
\end{eqnarray*}
It follows that $\|u\|_{\lp}\leq C \|u\|_{\Hsp}\,$, and this ends the proof.
\end{proof}

\smallskip

We now present two further embedding results for the spaces $\Hsp$. We need two preliminary lemmas.

\begin{lem} \label{3.10}
Let $\beta, s>0$ and $a\in\gbeta_s(\mathbb R^n)$. If $r>0$, then the
function $\varphi$ defined by
\begin{equation}\label{function varphi}
\varphi(x)=\dfrac{(1+|x|^2)^{r/2}}{(1+a(|x|^2))^{\frac{1}{2}(s+\frac{2r}{\beta})}}
\end{equation}
is a Fourier multiplier for $\lp$, $1 < p < \infty$.
\end{lem}
\begin{proof}
In this proof we use the multi-indexes conventions of \cite{Sai}.
It is clear that  the function $\varphi$ can be written as
\begin{equation}
\varphi(x)= (1+|x|^2)^{r/2}\,m_{a,s+\frac{2r}{\beta}}(x)  \; ,
\end{equation}
in which $m_{a,s+\frac{2r}{\beta}}$ is the function defined in (\ref{funcion m}) for $\mu = s+2r/\beta$.
Thus, for each multi-index $\alpha\leq(1,\ldots,1)$ we have
\begin{equation*}
x^{\alpha}D^{\alpha}(\varphi(x))=x^{\alpha}D^{\alpha}\left( (1+|x|^2)^{r/2}\,m_{a,s+\frac{2r}{\beta}}(x) \right)
\end{equation*}
and by Leibniz's formula (see Theorem 1.2 in \cite{Sai}) it follows that
\begin{equation*}
x^\alpha D^{\alpha}(\varphi(x)) =
x^\alpha \sum_{\gamma}D^{\gamma}\left((1+|x|^2)^{r/2}\right)\,D^{\alpha-\gamma}\left(m_{a,s+\frac{2r}{\beta}}(x)\right)\; ,
\end{equation*}
where $\gamma\leq \alpha$ and we have omitted constant numbers which are not essential for this proof. Then,
we have
\begin{eqnarray*}
\left| x^{\alpha}D^{\alpha}\left(\varphi(x)\right) \right| & = &
\left| x^\alpha \right| \left| \sum_{\gamma}x^{\gamma}\,(1+|x|^2)^{\frac{r}{2} - |\gamma|}\,
D^{\alpha-\gamma}\left(m_{a,s+\frac{2r}{\beta}}(x)\right) \right| \\
 & \leq & \left| x^\alpha \right|  \sum_{\gamma}\left| x^{\gamma}\right| \,(1+|x|^2)^{\frac{r}{2} - |\gamma|}\,
\left| D^{\alpha-\gamma}\left(m_{a,s+\frac{2r}{\beta}}(x)\right) \right| \\
 & \leq & \sum_\gamma (1+|x|^2)^{|\alpha|/2 + |\gamma|/2}\,(1+|x|^2)^{\frac{r}{2} - |\gamma|}
          \left| D^{\alpha-\gamma}\left(m_{a,s+\frac{2r}{\beta}}(x)\right) \right| \; .
\end{eqnarray*}
Now we use inequality (\ref{aux222}) with $k = |\alpha - \gamma|$ and $\mu = s+2r/\beta$ in order to bound the derivatives
$D^{\alpha-\gamma} \left( m_{a, s+\frac{2r}{\beta}} \right)$. We obtain, for $|x|$ large enough,
\begin{eqnarray}
\left| x^{\alpha}D^{\alpha}\left(\varphi(x)\right) \right| & \leq &
     \sum_\gamma (1+|x|^2)^{|\alpha|/2 + |\gamma|/2}\,(1+|x|^2)^{\frac{r}{2} - |\gamma|}
(1+|x|^2)^{-\frac{|\alpha-\gamma|}{2}-\frac{\beta}{4}\left( s+\frac{2r}{\beta}\right)+|\alpha-\gamma| \frac{\beta s}{4 n}}
\nonumber \\
 & = & \sum_\gamma (1+ |x|^2)^{-\frac{\beta s}{4} + |\alpha - \gamma|\frac{\beta s}{4 n}} \; . \label{aux333}
\end{eqnarray}
Since $|\gamma|\leq |\alpha|\leq n$, we obtain that the exponents appearing in (\ref{aux333}) are at most equal to
zero. Thus, $\left| x^{\alpha}D^{\alpha}\left(\varphi(x)\right) \right| \leq C$ for $|x|$ large enough and for all
multi-indexes $\alpha\leq (1,\ldots,1)$. Since the function $a$ satisfies $G_1$, we obtain $\varphi \in C^n(\Rn)$, and so we can bound
$\left| x^{\alpha}D^{\alpha}\left(\varphi(x)\right) \right|$ uniformly on $\Rn \setminus \{0\}$. It follows from Lemma \ref{gui} that
the function $\varphi$ is a Fourier multiplier for $\lp$.
\end{proof}

\smallskip

\begin{lem}
Let us consider the operator $\Lambda$ on $\lp$, $1 < p < \infty$, defined by
\begin{equation}
\Lambda(u)=\Fou^{-1}(\varphi(\xi)\Fou(u))\; ,
\end{equation}
where $\varphi$ is the function defined in $(\ref{function varphi})$. Then, there exists a positive constant $C$
such that
\begin{equation}
\|\Lambda (u)\|_{\lp}\leq C\|u\|_{\lp}\; .
\end{equation}
\end{lem}

\smallskip

This Lemma follows from the fact that $\varphi$ is a Fourier multiplier for
$\lp$. Our promissed embedding results for the spaces $\Hsp$ are:

\smallskip

\begin{teo}\label{embed on Hs}
Let $s , \beta > 0$ and $a\in\gbeta_s(\mathbb R^n)$ be fixed. Then, for each $r \geq 0$
the following continuous embedding
\begin{equation}
\mathcal{H}^{s+\frac{2r}{\beta},p}(a)\hookrightarrow H^{r,p}(\Rn) \; ,
\end{equation}
$1 < p < \infty$, holds, in which $H^{r,p}(\Rn)$ is classical fractional Sobolev space as defined for example in {\rm \cite{Tay3}}.
\end{teo}
\begin{proof}
Let us set $s_0 = s+\frac{2r}{\beta}$. Then, Proposition \ref{contencion gbetas}
implies that $a \in \gbeta_{s_0}(\mathbb{R}^n)$. We now observe that for $u\in \mathcal{H}^{s_0,p}(a)$ we have
\begin{eqnarray*}
\|u\|_{\hrp}&=&\|\Fou^{-1}\left( (1+|\xi|^2)^{r/2}\Fou(u)\right) \|_{\lp}\\
&=&\left\|\Fou^{-1}\left(\dfrac{(1+|\xi|^2)^{r/2}}{(1+a(|\xi|^2))^{s_0/2}}(1+a(|\xi|^2))^{s_0/2}\Fou(u)\right)\right\|_{\lp}\\
&=&\left\|\Lambda\left[ \Fou^{-1}\left( (1+a(|\xi|^2))^{s_0/2}\Fou(u) \right) \right]\right\|_{\lp}\\
&\leq&C\|\Fou^{-1}( (1+a(|\xi|^2))^{s_0/2}\Fou(u))\|_{\lp}\\
&=&C\|u\|_{\mathcal{H}^{s_{0},p}(a)} \; ,
\end{eqnarray*}
in which we have used that
$\Fou^{-1}\left( (1+a(|\xi|^2))^{s_0/2}\Fou(u) \right) \in \lp$ since
$u\in \mathcal{H}^{s_0,p}(a)$.
\end{proof}

\smallskip

\begin{cor}
Let $s , \beta > 0$ and $a\in\gbeta_s(\mathbb R^n)$ be fixed. If
$u \in \mathcal{H}^{s+\frac{2r}{\beta},p}(a)$, $1 < p < \infty$, and $r > n/p$, then the function $u$ is
continuous and bounded.
\end{cor}
\begin{proof}
The corollary follows from the above theorem and the fact that $H^{r,p}(\mathbb{R}^n) \subseteq C(\mathbb{R}^n) \cap L^\infty(\mathbb{R}^n)$,
see for instance Taylor's treatise \cite[Chp. 13, Proposition 6.3]{Tay3}.
\end{proof}

\smallskip

\begin{teo} \label{embed on Hs 2}
Let $s , \beta > 0$ and $a\in\gbeta_s(\mathbb R^n)$ be fixed. For each $\delta\geq0$ the following continuous embedding
holds:
\begin{equation} \label{h}
\mathcal{H}^{2s+\delta,p}(a)\hookrightarrow \Hsp\; , \quad \quad 1 < p < \infty \; .
\end{equation}
\end{teo}
\begin{proof}
As in the previous theorem, Proposition \ref{contencion gbetas}
implies that $a \in \gbeta_{2 s + \delta}\,$. Now we observe that for all
$u \in \mathcal{H}^{2s+\delta,p}(a)$ we have
\begin{eqnarray*}
\|u\|_{\Hsp}&=&\left\|\Fou^{-1}\left(\dfrac{(1+a(|\xi|^2))^{s/2}}{(1+a(|\xi|^2))^{\frac{1}{2}(2s+\delta)}}(1+a(|\xi|^2))^{\frac{1}{2}(2s+\delta)}
\Fou(u)\right)\right\|_{\lp}\\
&=&\left\|\Fou^{-1}\left(\dfrac{1}{(1+a(|\xi|^2))^{(s+\delta)/2}}(1+a(|\xi|^2))^{\frac{1}{2}(2s+\delta)}\Fou(u)\right)\right\|_{\lp}\\
&=&\left\|T_{s+\delta}\left(\Fou^{-1}\left[ (1+a(|\xi|^2))^{\frac{1}{2}(2s+\delta)}\Fou(u) \right]\right)\right\|_{\lp}\\
&\leq&C \left\| \Fou^{-1}\left[(1+a(|\xi|^2))^{\frac{1}{2}(2s+\delta)}\Fou(u)\right] \right\|_{\lp}\\
&=&C\|u\|_{\mathcal{H}^{2s+\delta,p}(a)} \; .
\end{eqnarray*}
\end{proof}

\section{Nonlinear Equations associated to $a(-\Delta)$}

The main aim of this section is the study in $\lp$ of the equation
\begin{equation}\label{ec lineal}
 ( a(-\Delta) + 1 )^{s/2}\, u = V(\cdot, u) \; ,
\end{equation}
in which $s > 0$ and the non-linearity $V$ satisfies some technical hypotheses to be specified below.
This equation encompasses several special cases of interest. If $s = 2$, Equation (\ref{ec lineal}) becomes
$$
( a(-\Delta) + 1 )\, u = V(u) \;
$$
so that, setting $a(-\Delta) + 1 = f(\Delta)$, we arrive at Equation (13)
of \cite{G-P-R}. Also, Lemma \ref{xxx} implies that special cases of Equation (\ref{ec lineal}) are equations
involving the fractional Laplace operator of the form
\begin{equation}\label{lapfrac}
 \left[ 1 + (-\Delta + m^2)^\gamma \right]^{s/2} \, u = V(\cdot, u) \; .
\end{equation}

Further examples of operators $a(-\Delta)$ of interest for Physics appear for instance in \cite{ABFS};
we will consider explicit examples of equation (\ref{ec lineal}) in Section 5.

\vspace{10pt}

\smallskip

We begin by solving the linear problem
\begin{equation}\label{ec lineal0}
 ( a(-\Delta) + 1 )^{s/2}\, u = g
\end{equation}
on the space $\Hsp\,$:

\begin{teo}\label{Teorema ec lineal}
Let $s , \beta > 0$ and $a\in\gbeta_s(\mathbb R^n)$. For each $g\in\lp$, $1 < p < \infty$, there exists an unique solution $u_g\in\Hsp$
to the linear equation $(\ref{ec lineal0})$.
Moreover, we have
\begin{equation}\label{igual normas ec lineal}
\|u_g\|_{\Hsp}=\|g\|_{\lp}\; .
\end{equation}
\end{teo}
\begin{proof}
Equation (\ref{ec lineal0}) is equivalent to
$$\Fou^{-1}\left( (1+a(|\xi|^2))^{s/2}\Fou(u) \right) = g \, ,$$
and therefore it is easy see that the solution $u_g$  is given by
 $$u_g=\Fou^{-1}\left(\dfrac{\Fou(g)}{(1+a(|\xi|^2))^{s/2}}\right)\; .$$
This solution belongs to $\Hsp$ by Theorem \ref{bi}, and (\ref{igual normas ec lineal}) holds.
\end{proof}

\smallskip

\smallskip

We are ready to study non-linear equations. First of all, we state the following
elementary proposition, showing that in some cases we can obtain existence,
{\em uniqueness} and regularity of solutions.

\begin{prop}\label{lema prop func V}
Let us assume that $1<p<\infty$. Let $\beta>2n/p ,\, s > 0$ and $a\in\gbeta_{s}$ be fixed.
Suppose further that the function $V:\Rn\times \mathbb R \tiend \mathbb C$ is such that $V(\cdot,0)\in\lp$,
and that there exists a function $h\in L^{\infty}(\Rn)$ such that
\begin{equation} \label{des}
|V(x,y_1)-V(x,y_2)|\leq h(x)|y_1-y_2|\; .
\end{equation}
\noindent
Then, for all $\delta > 0$ small enough, the equation
\begin{equation}\label{ec U_delta}
[1 + a(-\Delta)]^{s/2} (u) = \delta\, V(\cdot,u)
\end{equation}
has a unique solution $u\in\mathcal{H}^{s,p}(a)$.
\end{prop}
\begin{proof}
We define the function $R : \mathcal{H}^{s,p}(a) \rightarrow \mathcal{H}^{s,p}(a)$ which assigns to
each $u$ the unique solution $w_u$ of the linear problem $[1 + a(-\Delta)]^{s/2} (w_u) = \delta\, V(\cdot,u)$.
This function is well-defined: inequality (\ref{des}) implies
$$
|V(x,y)|  \leq  |V(x,y)-V(x,0)| + |V(x,0)| \leq h(x) |y| + |V(x,0)| \; ,
$$
and so $|V(x,y)|^p \leq C(p)\, ( h(x)^p |y|^p + |V(x,0)|^p )$ in which $C(p)$ is a constant number.
Our hypotheses on $h$ and $V(\cdot , 0)$ imply that $V(\cdot , u) \in \lp$ whenever
$u \in \lp$. Moreover, we observe that $R$ is a contraction for adequate choices of $\delta$. Indeed, first of all
we have:
$$
\|V(\cdot,u_1)-V(\cdot,u_2)\|_{\lp} \leq \|h\|_{L^{\infty}(\Rn)} \|u_1-u_2\|_{\lp}
\leq C \, \|h\|_{L^{\infty}(\Rn)} \, \|u_1-u_2\|_{\Hsp} \; ,
$$
because of the continuous embedding $\Hsp\hookrightarrow \lp$ proven in Theorem \ref{embed on Lp}. Now, using
Theorem \ref{bi} and (\ref{neweq}), we obtain
\begin{eqnarray*}
\|R(u_1) - R(u_2)\|_{\Hsp} & = & \|w_{u_1} - w_{u_2} \|_{\Hsp} \; = \; \|T_s^{-1}(w_{u_1} - w_{u_2}) \|_{\lp} \\
 & = & \delta\, \| V(\cdot,u_1) - V(\cdot,u_2) \|_{\lp} \\
 & \leq & \delta \, C \, \|h\|_{L^{\infty}(\Rn)} \, \|u_1-u_2\|_{\Hsp}  \; ,
\end{eqnarray*}
so that $R$ is a contraction for $\delta < 1/(2\,C \, \|h\|_{L^{\infty}(\Rn)})$. Existence and uniqueness of solution to (\ref{ec U_delta})
is then consequence of the Banach fixed point theorem.
\end{proof}

\begin{rem} \label{4.3}
Theorem $\ref{embed on Lp}$ implies that the solution $u$ obtained in the above proposition is in $\lp$. Even more,
we note that if we put $r = \beta/2\,$ in Theorem $\ref{embed on Hs}$ and we use Proposition $\ref{contencion gbetas}$, we obtain that
the solution $u$ belongs to the Sobolev space $H^{\beta/2 , p}(\Rn)$. Since $\beta/2 > n/p$ we obtain $H^{\beta/2,p}(\mathbb{R}^n)
\subseteq C(\mathbb{R}^n) \cap L^\infty(\mathbb{R}^n)$ {\rm (\cite[Chp. 13, Proposition 6.3]{Tay3})} and so $u$ is bounded and continuous.
\end{rem}

The next theorem establishes existence and regularity of solutions in a less
restrictive framework.

\begin{teo}\label{Teo ec no lineal}
Let $s , \beta > 0$ and $a\in\gbeta_s(\mathbb R^n)$ be given, and let $1<p<\infty$.
Given $\delta>0\,$, we consider the equation
\begin{equation} \label{nonlin1}
[1 + a(-\Delta)]^{s/2} (u) = \delta \varphi(x)V(x,u)
\end{equation}
in which $\varphi\in C^{\infty}_0$ and $V\in C^1(\Rn\times \mathbb R)$. Let us assume that there exist constant numbers $\alpha > 1$ and $C > 0$,
and a function $h\in\lp$ such that the function $V$ satisfies the following estimates
\begin{eqnarray}
\left| V(x,y) \right| + \left| \dfrac{\partial}{\partial x_i}V(x,y) \right|&\leq&C(h(x)+|y|^{\alpha})\hspace{1cm}i=1,2,\ldots,n \label{cond 1} \\
&&\nn\\
\left|\dfrac{\partial }{\partial y}V(x,y)\right|&\leq & C (1+|y|^{\alpha}).
                                                                 \label{cond 2}
\end{eqnarray}
If we take $m > 0$ and the parameters $\beta,s,m$ satisfy $m > n/(\alpha\, p)$ and
$s > 4m\alpha/\beta$,
then, for $\delta$ sufficiently small, Equation $(\ref{nonlin1})$
has a solution $u\in\Hsp$.
\end{teo}
\begin{proof}
Let us set $r_\alpha =\dfrac{n(\alpha -1)}{p\alpha}$.
We claim that if $u \in H^{r_\alpha ,\, p}(\Rn)$ is given, then the function
$ V(\cdot,u)\in \lp$. In fact, since $r_\alpha \, p < n$, the continuous inclusion
\begin{equation} \label{aste0}
H^{r_\alpha ,\, p}(\Rn) \hookrightarrow L^{\alpha p}(\Rn)
\end{equation}
holds (see for instance \cite[Chp. 13, Proposition 6.4]{Tay3}), and therefore
$$
\| u \|_{L^{\alpha\, p}(\Rn)} \leq C \| u \|_{H^{r_\alpha\, p}(\Rn)}
$$
for all $u \in H^{r_\alpha ,\, p}(\Rn)$. An easy calculation using (\ref{cond 1})
now yields
\begin{eqnarray}
\|V(\cdot,u)\|_{\lp}^p & \leq & C \left(\|h\|_{\lp}^p+
                     \int_{\Rn}\left|u(x)^\alpha\right|^p dx \right) \nonumber\\
                     & = & C \left(\|h\|_{\lp}^p+\|u\|_{L^{\alpha p}}^{\alpha\, p} \right) \, \label{h2}
\end{eqnarray}
and the claim follows. Since $\varphi \in L^\infty(\Rn)$, we also conclude that
the function $\varphi \, V(\cdot,u)$ belongs to $\lp$.

\smallskip

Now we need to use the following chain of continuous inclusions: we have (because of Theorem \ref{embed on Hs 2},
Theorem \ref{embed on Hs} and a result appearing in Lions' paper, \cite[p. 320]{Li})
\begin{equation} \label{h1}
\mathcal{H}^{s,p}(a) \hookrightarrow \mathcal{H}^{s/2 - \epsilon,p}(a) =
\mathcal{H}^{2(r_\alpha + m)/\beta,p}(a) \hookrightarrow H^{r_\alpha + m,\, p}(\Rn) \hookrightarrow H^{r_\alpha ,\, p}(\Rn) \; ,
\end{equation}
in which $m$ is as in the enunciate of the theorem and $\epsilon$ is determined by
the equation $s/2 - \epsilon = 2( r_\alpha + m)/\beta$. By the hypotheses of the
theorem, we obtain $\epsilon > 0$.

Next, we set
$A_0=\{u\in H^{r_\alpha +m,\, p}(\Rn) : \|u\|_{H^{r_\alpha +m,\, p}(\Rn)}\leq 1\}$
and we define the operator $\mathcal R:A_0\tien A_0$ as follows:
\begin{equation*}
\mathcal R(u)=w
\end{equation*}
in which $w$ is the unique solution to the linear equation
\begin{equation}\label{ec L2 u=d phi V}
[1 + a(-\Delta)]^{s/2} w = \delta\varphi V(\cdot,u)\; .
\end{equation}
Since $\delta \varphi V(\cdot,u)\in\lp$, Theorem \ref{Teorema ec lineal} implies
that there exists an unique solution $w$ to the equation (\ref{ec L2 u=d phi V})
and therefore $\mathcal R$ is well-defined. We now check that its range is indeed $A_0$ if
we choose $\delta$ appropriately.

Since $a$ belongs to $\gbeta_s(\mathbb R^n)$, Theorem \ref{Teorema ec lineal} tells us that the solution
$ w = {\mathcal R}(u)$ belongs to $\Hsp$. Inclusions (\ref{h1}) imply
that $w$ belongs to $H^{r_\alpha+m,p}(\Rn)$. We have that
$$\| {\mathcal R}(u) \|_{H^{r_\alpha +m ,\, p}(\Rn)} = \|w\|_{H^{r_\alpha+m,\, p}(\Rn)}\; , $$
and then the following inequalities hold:
\begin{eqnarray}
\|w\|_{H^{r_\alpha +m,\, p}(\Rn)} &\leq & C\|w\|_{\Hsp} \nonumber \\
&=&C\|\delta\varphi V(\cdot,u)\|_{\lp} \nonumber \\
&\leq &C\delta\|\varphi\|_{\linf}\left(\|h\|_{\lp}^p +
\widetilde C\|u\|_{H^{r_\alpha +m ,\, p}(\Rn)}^{\alpha p} \right) \; , \label{aste}
\end{eqnarray}
in which we have used (\ref{h2}) and the inclusions (\ref{h1}). Since $u\in A_0$, we conclude that
\begin{equation*}
\|\mathcal R(u)\|_{H^{r_\alpha +m,\, p}(\Rn)}\leq C\delta\|\varphi\|_{\linf}\left( \|h\|_{\lp}^p +\widetilde C\right).
\end{equation*}
Hence, since the right side of the above inequality does not depend on $u$, there exists a sufficiently small $\delta$ such that for all $u\in A_0$
\begin{equation*}
\|\mathcal R(u)\|_{H^{r_\alpha +m,\, p}(\Rn)}\leq 1 \; ,
\end{equation*}
and so the operator $\mathcal R$ is well defined.

\smallskip

We now show that the operator $\mathcal R$ has a fixed point on $A_0$.
We use the Schauder fixed point Theorem.

First we check continuity of $\mathcal R$. let us note that
\begin{eqnarray*}
|V(x,u_1(x))-V(x,u_2(x))|&=&\left|\int_0^1\dfrac{d}{dt}\left(V(x,tu_1(x)+(1-t)u_2(x))\right)dt\right|\\
&=&\left|\int_0^1 D_y V(x,tu_1(x)+(1-t)u_2(x))(u_1(x)-u_2(x))dt\right|\\
&\leq&|u_1(x)-u_2(x)|\int_0^1\left|D_y V(x,tu_1(x)+(1-t)u_2(x))dt\right|\\
&\leq&|u_1(x)-u_2(x)|\, C\int_0^1|1+|tu_1(x)+(1-t)u_2(x)|^{\alpha}|dt\\
&\leq&C\, |u_1(x)-u_2(x)|\int_0^1 |1+t|u_1(x)|^{\alpha}+(1-t)|u_2(x)|^{\alpha}|dt\\
&\leq& C\, |u_1(x)-u_2(x)|\int_0^1 ( 1+|u_1(x)|^{\alpha}+|u_2(x)|^{\alpha} ) dt\\
&=& C\, |u_1(x)-u_2(x)|\, (1+|u_1(x)|^{\alpha}+|u_2(x)|^{\alpha}) \; .
\end{eqnarray*}
We use this observation to estimate the difference $\varphi\, V(\cdot,u_1)-\varphi\,V(\cdot,u_2)$.
We make use of the fact that $\varphi$ has compact support, let us say $K \subset \Rn$.
\begin{eqnarray*}
\|\varphi\, V(\cdot,u_1)-\varphi\,V(\cdot,u_2)\|_{\lp}^p&=&\int_{K}|\varphi(x)|^p\, |V(x,u_1(x))-V(x,u_2(x))|^p dx\\
&\leq&C\,\| \varphi \|_{\linf} \int_{K}|u_1(x)-u_2(x)|^p\left(1+|u_1(x)|^{\alpha}+|u_2(x)|^{\alpha}\right)^p dx\\
&\leq& C\,\| \varphi \|_{\linf} \int_{K}|u_1(x)-u_2(x)|^p\left(1+|u_1(x)|^{\alpha\, p}+|u_2(x)|^{\alpha\, p}\right) dx\\
& \leq & C\,\| \varphi \|_{\linf} \| u_1 -u_2 \|^p_{L^p(K)} \left(1 + \|u_1\|^{\alpha\, p}_{\linf} +\|u_2\|^{\alpha\, p}_{\linf} \right) \; ,
\end{eqnarray*}
in which $C$ is a generic constant and we have used that $u_1, u_2 \in H^{r_\alpha + m , p}(\Rn)$ implies
$u_1, u_2 \in \linf$ (indeed, the estimate on $m$ appearing in the hypothesis of the theorem implies that $r_\alpha + m > n/p$, and therefore
$H^{r_\alpha + m , p}(\Rn) \subset C(\Rn) \cap \linf$, as explained in Remark \ref{4.3} following \cite{Tay3}).

Hence, we obtain the following inequalities, in which we have used the continuous inclusions (\ref{h1}):
\begin{eqnarray*}
\|\mathcal{R}(u_1)-\mathcal{R}(u_2)\|_{H^{r_\alpha +m,\, p}(\Rn)}&=&
\| w_1 - w_2 \|_{H^{r_\alpha +m,\, p}(\Rn)} \; \leq \;  C \, \| w_1 - w_2 \|_{\mathcal{H}^{s,p}(a)} \\
& = & C \, \|\delta\, \varphi\, \left( V(\cdot,u_1)-V(\cdot,u_2) \right) \|_{\lp} \\
&\leq&C \, \delta  \| \varphi \|_{\linf} \| u_1 -u_2 \|^p_{L^p(K)} \left(1 + \|u_1\|^{\alpha\, p}_{\linf} +\|u_2\|^{\alpha\, p}_{\linf} \right) \; .
\end{eqnarray*}
The second equality holds because of (\ref{igual normas ec lineal}) and
the fact that $v = w_1 - w_2$ is the solution to the linear problem
$[1 + a(-\Delta)]^{s/2} (v) =  V(\cdot,u_1)-V(\cdot,u_2)\,$.

\smallskip

\smallskip

Now we use that $\| u_1 -u_2 \|^p_{L^p(K)} \leq C \| u_1 -u_2 \|^p_{L^{\alpha p}(K)}$
---because $p < \alpha\,p$, since $\alpha > 1$--- and also that $\| u_1 -u_2 \|^p_{L^{\alpha p}(K)} \leq \| u_1 -u_2 \|^p_{L^{\alpha p}(\Rn)}$.
Thus, the continuous inclusions (\ref{h1}) and (\ref{aste0}) imply
\[
\| u_1 -u_2 \|^p_{L^p(K)} \leq C_1 \| u_1 -u_2 \|^p_{L^{\alpha p}(\Rn)} \leq
C_2 \| u_1 -u_2 \|^p_{H^{r_\alpha +m,\, p}(\Rn)}\; .
\]
We conclude that if $\| u_1 -u_2 \|^p_{H^{r_\alpha +m,\, p}(\Rn)} \tien 0$ then
$\|\mathcal{R}(u_1)-\mathcal{R}(u_2)\|_{H^{r_\alpha +m,\, p}(\Rn)}\tien0$. This fact shows the continuity of the operator $\mathcal{R}$.

\smallskip

\smallskip

Now we show that  $\mathcal{R}$ is a compact operator. Let us consider a bounded sequence $\{u_k\}$  in $H^{r_\alpha + m,p}(\Rn)$.
We will show that the sequence $\{\mathcal{R}(u_k)\}$ has a convergent subsequence in $H^{r_\alpha + m,p}(\Rn)$. Since
$\varphi\in C^\infty_0(\Rn)$, there exists $R>0$ such that $supp(\varphi)\subset B(0,R)$. Now, for each $k\in\mathbb{N}$ let us define
$$g_k(x)=\left\{\begin{array}{cc}
\delta\varphi(x)V(x,u_k(x))&\mbox{ if }x\in B(0,R)\\
0&\mbox{ if }x\in\Rn\setminus B(0,R).
\end{array}\right.$$
We check that the sequence $\{g_k\}_{k \in \mathbb{N}}$ is bounded in
$H^{1,p}(\Rn)$.
We have,
\begin{eqnarray*}
\|g_k\|_{H^{1,p}(\Rn)}&=&\delta\|\varphi V(\cdot,u_k)\|_{H^{1,p}(B)}\\
&=&\delta\left(\|\varphi V(\cdot,u_k)\|_{L^{p}(B)}+\sum_{i=1}^n\|\, D_i \,[ \varphi V(\cdot,u_k) ] \, \|_{L^{p}(B)}\right)^{1/p} \; .
\end{eqnarray*}
Since for each $k$ we have $u_k\in\ H^{r_\alpha + m,p}(\Rn)$, inequality (\ref{aste}) implies
 $$\|\varphi V(\cdot,u_k)\|_{L^{p}(B)}<\infty \; ,$$
and therefore we only need to show that  $\sum_{i=1}^n\|D_i [\varphi V(\cdot,u_k)] \|_{L^{p}(B)}<\infty$. By the chain rule for weak derivatives,
see \cite{Ev}, we have
\begin{eqnarray*}
D_i\left[\varphi(x)V(x,u_k(x))\right]&=&D_i(\varphi(x))V(x,u_k(x))+  \varphi(x)D_i\left(V(x,u_k(x))\right)\\
& = & D_i(\varphi(x))V(x,u_k(x))+ \\
&   & \varphi(x)\left[D_i V(x,u_k(x))+ D_yV(x,u_k(x))D_iu_k(x)\right]
\end{eqnarray*}
for $i=1\ldots n$, and so,
\begin{eqnarray*}
&&\|D_i\varphi V(\cdot,u_k)\|_{L^p(B)}^p\\
&=&\int_{B}\left|D_i(\varphi(x))V(x,u_k(x))+\varphi(x)\left[D_i V(x,u_k(x))+D_yV(x,u_k(x))D_iu_k(x)\right]\right|^p dx\\
&\leq& C(p)\left[ \int_{B}\left|D_i(\varphi(x))V(x,u_k(x))\right|^p dx+\int_{B}|\varphi(x)|^p|D_i V(x,u_k(x))+D_yV(x,u_k(x))D_iu_k(x)|^p\right]\\
&\leq&C(p)\, \|D_i (\varphi)\|_{L^\infty(B)}^p \|V(\cdot,u_k)\|_{L^p(B)}^p +\\
&&\left. C_1(p)\|\varphi\|_{L^\infty(B)}^p\left(\|D_i V(\cdot,u_k)\|_{L^p(B)}^p+\|D_yV(\cdot,u_k)D_i u_k\|_{L^p(B)}^p\right)\right. .
\end{eqnarray*}
Since (\ref{cond 1}) holds, we have
\begin{eqnarray}
\|D_i V(\cdot,u_k)\|_{L^p(B)}^p&=&\int_{B}|D_i V(x,u_k(x))|^pdx\nn\\
&\leq&C\int_{B}|h(x)|^p+|u_k(x)|^{\alpha p}\nn\\
&=&C\left(\|h\|_{L^p(B)}^p+\|u_k\|_{L^{\alpha p}}^{\alpha p}\right)\nn\\
&\leq& C\left(\|h\|_{L^p(B)}^p+\widetilde{C}\|u_k \|_{H^{r_\alpha +m,p}(B)}^{\alpha}\right),
\end{eqnarray}
On the other hand, we note that by (\ref{cond 2})
\begin{eqnarray}
\|D_y V(\cdot,u_k)D_i u_k\|_{L^p(B)}^p&=&\int_{B}|D_y V(x,u_k(x))D_i u_k(x)|^pdx\nn\\
&\leq&C\int_{B} \left| (1+|u_k(x)|^{\alpha})D_iu_k(x) \right|^p dx  \nn\\
&=&C\int_{B} \left| D_i u_k(x)+|u_k(x)|^{\alpha}D_iu_k(x) \right|^pdx \nn\\
&\leq & C_2(p)\left(\|D_i u_k\|_{L^p(B)}+\||u_k|^{\alpha p}\|_{L^{\infty}(B)}\|D_i u_k\|_{L^p(B)}^p\right)\nn\\
&=& C_2(p)\|D_i u_k\|_{L^p(B)}^p\left(1+\|\,|u_k|^{\alpha p}\|_{L^{\infty}(B)}\right).\label{Dy V en Lp}
\end{eqnarray}
Since the sequence $\{u_k\} \subset H^{r_\alpha + m,p}(\mathbb{R}^n)$ is bounded, there exists
$M > 0$ so that  $$\|D_i u_k\|_{L^{p}(B)}^p \leq M$$ for all $k$. Moreover, as already noticed above, $H^{r_\alpha + m,p}(\mathbb{R}^n) \subset
C(\Rn) \cap L^\infty(\Rn)$, and therefore we can conclude that
$\left(1+\||u_k|^{\alpha p}\|_{L^{\infty}(B)}\right)<\infty\,$ uniformly in $k$.
Hence by (\ref{Dy V en Lp}) we have
$$\|D_y V(\cdot,u_k)D_i u_k\|_{L^p(B)}<\infty.$$

Thus, we have proven that $\{g_k\}_{k \in \mathbb{N}}$ is a bounded sequence in $H^{1,p}(B)$. By the Rellich-Kondrachov
theorem (see for instance \cite[p. 274]{Ev}) the embedding $H^{1,p}(B)\hookrightarrow L^p(B)$ is compact. Hence, there exists
a subsequence $\{g_{k_i}\}$ of $\{g_k\}_{k \in \mathbb{N}}$ which converges in
$L^{p}(B)$.

This fact allows us to show that the sequence $\{w_{k_i}\}=\mathcal{R}(u_{k_i})$ is a Cauchy sequence  in
$H^{r_\alpha + m,p}(\mathbb{R}^n)$. Indeed, we use the embedding $\Hsp\hookrightarrow H^{r_\alpha + m,p}(\mathbb{R}^n)$,
see (\ref{h1}). Then, it follows that:
\begin{eqnarray*}
\|w_{k_i}-w_{k_j}\|_{H^{r_\alpha + m,p}(\Rn)}&\leq&C\|w_{k_i}-w_{k_j}\|_{\Hsp}\\
&=&C\|\delta\varphi V(\cdot,u_{k_i})-\delta\varphi V(\cdot,u_{k_j})\|_{\lp}\\
&\leq&C\|g_{k_i}-g_{k_j}\|_{L^p(B)}.
\end{eqnarray*}
Hence, the sequence $\{\mathcal{R}(u_{k_i})\}$ is convergent in
the Banach space $H^{r_\alpha + m,p}(\mathbb{R}^n)$. We have proven that the operator
$\mathcal{R}$ is compact. By Schauder's theorem, there
exists at least one fixed point $u_0$ of $\mathcal{R}$, and hence
there exists a solution in $\Hsp$ to the equation (\ref{nonlin1}).
\end{proof}

\smallskip

\smallskip

We finish this section with an existence proof in the radial case. Our main technical
references for this part of the paper is Lions' classic paper \cite{Li} and the recent treatise \cite{Tri}. Let us assume that $t$ and $p$
are real numbers such that $\,t\, p > n$. Then,
Proposition 6.3 of \cite{Tay3} tells us that $H^{t,p}(\Rn) \subset C(\Rn) \cap L^\infty(\Rn)$. We define, after \cite{Li},
$$H^{t,p}_r(\Rn) = \{ u \in H^{t,p}(\Rn) : u \mbox{ is spherically symmetric} \}\; .$$
Moreover, if $\beta\,s\, p > 4 n$ and $a \in \gbeta_s(\mathbb R^n)$, we define the following
closed subspace of $\mathcal{H}^{s,p}(a)$:
$$ \mathcal{H}^{s,p}_r(a) = \{ u \in \mathcal{H}^{s,p}(a) : u \mbox{ is spherically symmetric} \}\; .$$
This definition makes sense because Theorem \ref{embed on Hs} implies that if $\beta\,s\, p > 4 n$, then $\Hsp$ is contained
in  $C(\Rn) \cap L^\infty(\Rn)$. We also need the standard definition
$L^{p}_r (\Rn) = \overline{C^\infty_{0,r}(\Rn)}^{\lp}\,$.

\smallskip

\smallskip

We begin by stating the following corollary of Theorem \ref{Teorema ec lineal}:

\begin{cor} \label{radlin}
Let $\beta>0$, $s > 0$ and $a\in\gbeta_s(\mathbb R^n)$ be fixed. If $g\in\lp$ is a spherically symmetric function,
then the solution to the linear equation $(\ref{ec lineal0})$ is also spherically symmetric.
\end{cor}
\begin{proof}
We have that the solution to equation (\ref{ec lineal0}) is given by
$$u_g=\Fou^{-1}\left(\dfrac{\Fou(g)}{(1+a(|\xi|^2))^{s/2}}\right)\; .$$
Since $g$ is a spherically symmetric function, then $\Fou (g)$ is a spherically symmetric tempered distribution (see for
instance \cite[p. 148]{Kan}), and therefore $\Fou(g)/(1+a(|\xi|^2))^{s/2}$ is a spherically symmetric tempered distribution as well.
Hence, $u_g$ is spherically symmetric.
\end{proof}

Now we state a remark on continuous inclusions. Assertions (\ref{h22}) and (\ref{inclusion}) below will be crucial
for the enunciate and proof of the next theorem.

\begin{rem} \label{4.6}
Let us assume that $\alpha>1$, $\beta>0$, and suppose that $a\in\gbeta_s(\mathbb R^n)$.
Let us set $r_\alpha = n(\alpha -1)/(\alpha p)$ as in Theorem $4.3$. Then, we have the inequalities
$$p < \alpha p < \frac{p\, n}{n - p\, r_\alpha}$$ for $n > 1$, and we can use the embedding theorems on radial functions
appearing in Lions' paper {\rm \cite{Li}}, see also {\rm \cite[Section 6.5.2]{Tri}}. We obtain, using also $(\ref{h1})$,
the chain of continuous inclusions
\begin{equation} \label{h22}
\mathcal H^{s,p}_r(a) \hookrightarrow H_r^{r_\alpha,\,p}(\Rn) \hookrightarrow
\hookrightarrow L^{\alpha p}_r (\Rn)\; ,
\end{equation}
in which ``$\hookrightarrow \hookrightarrow$" denotes compact embedding. In particular, there exists a constant
number $N$ depending on $s,p,\alpha$ such that
\begin{equation}
\| u \|_{L^{\alpha p}(\Rn)} \leq N\, \| u \|_{\mathcal H^{s,p}_r(a)} \;   \label{inclusion}
\end{equation}
for all $u \in \mathcal H^{s,p}_r(a)$.
\end{rem}

\smallskip

\noindent In the theorem below the constant number $N$ appearing therein is the fixed number exhibited in (\ref{inclusion}).

\smallskip

\begin{teo} \label{radi}
Let us assume that $\alpha>1$, $\beta>0$,
suppose that $a\in\gbeta_s(\mathbb R^n)$, and that $V(x,y)$ is spherically symmetric with respect to $x$.
Assume also that there exist functions $h\in \lp$ and
$g\in L^{\frac{\alpha p}{\alpha-1}}(\Rn)$ such that the following two inequalities hold:
\begin{equation}\label{cond2 no lineal radial}
|V(x,y)|\leq C(|h(x)|+|y|^{\alpha}),\quad\quad \left|\dfrac{\partial }{\partial y} V(x,y) \right|\leq C(|g(x)|+|y|^{\alpha-1})
\end{equation}
for some constant $C>0$. Then, if we choose
$$\epsilon > (2^p C^p N)^{1/(1-\alpha)} \mbox{ and } \rho_{\epsilon} = \epsilon/(2^p C^p N) - \epsilon^\alpha\; ,$$
there is a radial (i.e. spherically symmetric) solution $u\in\Hspr$ to the equation
\begin{equation}\label{eq_radial}
[1 + a(-\Delta)]^{s/2} u = V(\cdot, u)
\end{equation}
with $\|u\|_{L^{\alpha p}_r(\Rn)}\leq \epsilon\,$ whenever
$\|h\|_{L^{p}_r(\Rn)} < \rho_{\epsilon}$.
\end{teo}
\begin{proof}
First of all we note that conditions (\ref{cond2 no lineal radial}) imply that if $u \in L^{\alpha p}_r(\Rn)$,
then the function $V( \cdot , u ) \in L^p_r(\Rn)$:
\[
\left|V(x,u(x))\right|^p
\leq C^p \left(|h(x)|+|u(x)|^{\alpha}\right)^p
\leq 2^p C^p (|h(x)|^p+|u(x)|^{\alpha p}) \; ,
\]
and therefore
\begin{equation}\label{ineq V_u}
\|V(\cdot , u) \|_{\lp}\leq 2^p C^p \left(\|h\|_{\lp}+\|u\|^\alpha_{L^{\alpha p}(\Rn)} \right) \; ,
\end{equation}
so that $V(\cdot , u) \in L^{p}(\Rn)$. Our hypotheses imply
that, moreover, $V(\cdot , u)$ is radial for $u\in L^{\alpha p}_r(\Rn)$.

This observation allows us to define the function $\mathcal
G:X_\epsilon\tien L^{\alpha p}_r(\Rn)$, in which $X_\epsilon$ is
 the ball $X_\epsilon=\{u\in
L^{\alpha p}_{r}(\Rn):\|u\|_{L^{\alpha p}(\Rn)}\leq \epsilon\}$,
as follows:
$$\mathcal G (u)=\tilde u \; , $$
in which $\tilde u$ is the unique solution to the linear equation
$$[1 + a(-\Delta)]^{s/2} \tilde u= V(\cdot , u)\; .$$
By Theorem \ref{Teorema ec lineal} and Corollary \ref{radlin}, we conclude that
$\mathcal{G}(u) = \tilde u \in \mathcal H^{s,p}_r(a)$ and that
$\| \tilde{u} \|_{\mathcal H^{s,p}_r(a)} = \| V( \cdot , u) \|_{\lp}$.
We notice that the map $\mathcal{G}$ is well defined
since inclusions (\ref{h22}) imply that $\tilde{u} \in L^{\alpha p}_r (\Rn)$.

Now we claim that there exists $\epsilon > 0$ such that $\mathcal
G:X_\epsilon\tien X_\epsilon\,$. Indeed, the continuous inclusion
(\ref{inclusion}) implies that
 \begin{equation}\label{ineq Gu Vu}
 \|\mathcal G(u)\|_{L^{\alpha p}(\Rn)} \leq N\,  \|\mathcal G(u)\|_{\mathcal H^{s,p}_r(a)} =
 N\, \|V(\cdot , u) \|_{\lp} \; ,
 \end{equation}
in which we have used Theorem \ref{Teorema ec lineal}, and
therefore by inequality (\ref{ineq V_u}) we obtain
 \begin{equation}
 \|\mathcal G(u)\|_{L^{\alpha p}(\Rn)} \leq 2^p C^p N (\|h\|_{\lp}+\|u\|^\alpha_{L^{\alpha p}(\Rn)})
 \leq 2^p C^p N (\|h\|_{\lp}+\epsilon^\alpha) \; .
 \end{equation}
Thus, if we choose $\epsilon > (2^p C^p N)^{1/(1-\alpha)}$ as in the enunciate of the theorem, and
we assume that $\|h\|_{\lp} < \rho_\epsilon = \epsilon/(2^p C^p N) - \epsilon^\alpha$, we
obtain $\|\mathcal G(u)\|_{L^{\alpha p}(\Rn)} \leq \epsilon$ and
so $\mathcal G:X_\epsilon\tien X_\epsilon\,$, as claimed.

\smallskip \smallskip

As in Theorem \ref{Teo ec no lineal}, we plan to apply the Schauder fixed point
theorem to the map $\mathcal{G}$. First, we show that the function
$\mathcal G$ is continuous:

Let us consider a sequence $\{u_n\}\subset X_\epsilon$ such that
$u_n\tien u$ in $L^{\alpha p}(\Rn)$. We set $\mathcal{G}(u_n) =
\tilde{u}_n$ and $\mathcal{G}(u) = \tilde{u}$. We estimate
$\|\tilde u_n-\tilde{u}\|_{L^{\alpha p}(\Rn)}$ as follows:

The continuous inclusions (\ref{h22}) yield
\begin{equation}
\label{aux3} \|\tilde u_n-\tilde{u}\|^p_{L^{\alpha p}(\Rn)}\leq C
\|\tilde u_n-\tilde u\|^p_{\mathcal H^{s,p}_r(a)} = C \|V(\cdot ,
u_n) - V(\cdot , u)\|_{\lp} \; ,
\end{equation}
and we can estimate this difference using the fundamental theorem
of calculus and hypotheses (\ref{cond2 no lineal radial}):
\begin{eqnarray*}
&&\left|V(x , u_n(x))-V(x,u(x)) \right|\\
&=&\left|\int_0^1\dfrac{d}{dt}\left[V(x,t u_n(x)+(1-t)u(x))\right]dt\right|\\
&=&\left|(u_n(x)-u(x))\int_0^1\left(\dfrac{\partial}{\partial y}\left[V(x,t u_n(x)+(1-t)u(x))\right]\right)dt\right|\\
&\leq& C\left|u_n(x)-u(x) \right| \int_0^1 \left(|g(x)|+|tu_n(x) + (1-t) u(x)|^{\alpha-1}\right)dt\\
&\leq&C_1 \left|u_n(x)-u(x)|\left(|g(x)|+|u_n(x)|^{\alpha-1}+|u(x)|^{\alpha-1}\right)\right|\; ,
\end{eqnarray*}
and therefore by H\"older inequality we have
\begin{eqnarray*}
\lefteqn{\|V(\cdot , {u_n}) - V(\cdot , u) \|^p_{L^p(\Rn)}
\leq}\\&& C_1 \left( \int |u_n(x)-u(x)|^{\alpha p}
\right)^{1/\alpha} \left( \int
[\,|g(x)|+|u_n(x)|^{\alpha-1}+|u(x)|^{\alpha-1}]^{\alpha p/(\alpha
-1)}\right)^{(\alpha-1)/(\alpha p)} \; .
\end{eqnarray*}
It follows that
\begin{eqnarray*}
\lefteqn{\|V(\cdot , {u_n}) - V(\cdot , u) \|^p_{L^p(\Rn)} \leq}\\
&& C_2 \| u_n - u \|_{L^{\alpha p}(\Rn)}^p \left[ \int
|g(x)|^{\alpha p/(\alpha -1)} dx + \int ( | u_n(x)|^{\alpha p} +
|u(x)|^{\alpha p} )dx \right]^{(\alpha-1)/(\alpha p)} \; .
\end{eqnarray*}

Since $u_n\tien u$  in $L^{\alpha p}(\Rn)$ and $g\in
L^{\frac{\alpha p}{\alpha-1}}(\Rn)$ by hypothesis, inequality
(\ref{aux3}) implies that $\tilde u_n\tien\tilde u$ in $L^{\alpha
p}(\Rn)$, so that $\mathcal{G}$ is continuous, as claimed.

Now we prove that $\mathcal{G}$ is compact. We use once more the
inclusions (\ref{h22}). Let $(u_n)_{n \in \mathbb{N}} \subset
X_\epsilon$ be a bounded sequence, so that $\| u_n \|_{L^{\alpha
p}(\Rn)} \leq M$ for all $n \in \mathbb{N}$. We have,
\begin{eqnarray*}
\| \mathcal{G} (u_n) \|_{H_r^{r_\alpha,\,p}(\Rn)} & \leq & C \|
\mathcal{G} (u_n) \|_{\mathcal{H}_r^{s,p}(a)} \\
 & = & C \| V( \cdot , u_n) \|_{\lp} \\
& \leq & C C_1 \left(\|h\|_{\lp}+\|u_n \|^\alpha_{L^{\alpha
p}(\Rn)} \right) \; ,
\end{eqnarray*}
in which we have used (\ref{ineq V_u}) in the last inequality.
Thus, the sequence $( \mathcal{G} (u_n) )_{n \in \mathbb{N}}$ is
bounded in $H_r^{r_\alpha,\,p}(\Rn)$. Since the last continuous
inclusion in (\ref{h22}) is compact, we conclude that this
sequence has a convergent subsequence in $X_\epsilon$ with respect
to the topology of $L^{\alpha p}(\Rn)$.

In conclusion, the map $\mathcal G:X_\epsilon\tien X_\epsilon$ is
compact and continuous. By Schauder's fixed point theorem, we have
that $\mathcal{G}$ possesses a fixed point. By Theorem
\ref{Teorema ec lineal} this fixed point belongs to $\Hspr$, and
therefore, there exists a radial solution to the nonlinear
equation (\ref{eq_radial}) in $\Hspr$.
\end{proof}

\section{An Example: the fractional Laplace operator}

In this section we are interested in studying the existence of
radial solutions to the non-linear equation (\ref{eq_radial}) if
$a_m(-\Delta) = (-\Delta + m^2)^{\gamma/2}$, $\,0 < \gamma < 1$. The $\gamma=1$ case of this operator
appears for instance in Dubinskii's paper \cite{Du} and in the more recent works \cite{va,FJL,L}, in which they consider
(evolutionary versions of) equations of the form
$$
\sqrt{(-\Delta) + m^2 }\, u + f(u) = 0 \; .
$$
In addition to these equations, we mention as motivational examples
the fractional non-linear Schr\"odinger equation
\begin{equation} \label{fls}
\left[ (- \Delta + m^2)^{s} u - m^{2s} \right] u +\mu u = |u|^{p-2} u \;  ,
\end{equation}
considered in \cite{va2} under suitable assumptions on the real parameters $m,\mu,s,p$. Also important for us are the Benjamin-Ono equation
\begin{equation} \label{bo}
(- \Delta)^{\gamma/2} u = u^2 - u \;  ,
\end{equation}
and the (deformed) Peierls-Nabarro equation
\begin{equation} \label{pn}
(- \Delta)^{\gamma /2} u = - \kappa u + d(|x|)\,\sin (u) \;  ,
\end{equation}
in which for simplicity we assume that $d$ is a continuous function of compact support.
In addition to the papers just cited, we mention \cite{CS1,CS2} by Cabr\'e and Sire, and \cite{Sir} by Sire and
Valdinoci, on equations of the form  $(-\Delta)^{\gamma} u = f(u)$,
and the papers \cite{FL,FLS} by Frank and his coworkers on scalar and vector equations of the form
$(- \Delta)^{\gamma} u + u - |u|^p\, u = 0$.
We also highlight the classical paper \cite{AT}, in which the authors study Equation (\ref{bo}) for $\gamma = 1$,
and explain its relationship with the evolutionary version of the Benjamin-Ono equation of soliton theory, and the
numerical investigation of the fractional non-linear Schr\"odinger equation carried out in \cite{KSM}.

\smallskip

Now let us recall that Lemma \ref{xxx} tells us that $a_m(t)  = (|t| + m^2)^{\gamma/2}$, $0<\gamma<1$ and $m \neq 0$,
belongs to the class $\mathcal{G}^\gamma_s$ for $s > 4 n/\gamma\,$; therefore, it makes sense to apply Theorem \ref{radi}
to equations of the form
\begin{equation} \label{gp}
\left[\,1 + (- \Delta + m^2)^{\gamma /2}\,\right]^{s/2}\,u = V(\cdot , u) \; , \quad \quad m \neq 0 \; ,
\end{equation}
for an appropriate function $V(\cdot , u)$. {\em We note that the condition $s > 4 n/\gamma$ implies that we cannot
set $s=2$ in Equation {\rm (\ref{gp})} for $0<\gamma<1$}. We obtain the following proposition, in which we use notation introduced in Remark \ref{4.6}:

\begin{prop} \label{fin}
Let us assume that the right hand side of $(\ref{gp})$ satisfies conditions
\begin{equation}\label{cond2 no lineal radial 2}
|V(x,y)|\leq C(|h(x)|+|y|^{\alpha})\, , \quad\quad \left|\dfrac{\partial }{\partial y} V(x,y) \right|\leq C(|g(x)|+|y|^{\alpha-1})
\end{equation}
for a fixed $\alpha > 1$. Then, if we choose
$$ \epsilon > (2^p C^p N)^{1/(1-\alpha)} \quad \mbox{ and } \quad \rho_{\epsilon} = \epsilon/(2^p C^p N) - \epsilon^\alpha \; ,$$
there is a radial solution $u\in\mathcal{H}^{s,p}_r(a_m)$ to Equation $(\ref{gp})$ whenever $\|h\|_{L^{p}_r(\Rn)} < \rho_{\epsilon}$.
\end{prop}

\noindent
Inclusions (\ref{h22}) imply that the solution $u$ belongs to the radial fractional
Sobolev space $H^{r_\alpha , \,p}_r(\Rn)$, in which $r_\alpha = n(\alpha -1)/(\alpha p)$, and to $L^{\alpha\, p}_r(\Rn)$.

\smallskip

Proposition \ref{fin} applies, in particular, to ``perturbed" equations which we now introduce, motivated by versions of
the (focusing) pseudo-relativistic Allen-Cahn type Equation
\begin{equation} \label{ac}
(- \Delta + m^2)^{\gamma/2} u = -u + u^3 \;
\end{equation}
and (\ref{fls}), see for instance \cite{MR,WW1,WW2}. We use the same notation as in the previous proposition,
and for simplicity we ``perturb" via continuous radial functions of compact support.

\begin{cor}
There exists a spherically symmetric solution $u\in\mathcal{H}^{s,p}_r(a_m)$ to the equations
\begin{equation*}
\left[ \,1 + (-\Delta + m^2)^{\gamma/2}\,\right]^{s/2} u = \kappa\, u^3 + \rho(|x|) \; , \quad \quad m \neq 0 \; ,
\end{equation*}
and
\begin{equation*}
\left[ \,1 + (-\Delta + m^2)^{\gamma/2}\,\right]^{s/2} u = |u|^\beta u + \rho(|x|) \; , \quad \quad m \neq 0 \; ,
\end{equation*}
in which $\rho$ is a continuous radial function of compact support such that $\|\rho\|_{L^{p}_r(\Rn)} < \rho_{\epsilon}$.
\end{cor}

\begin{rem} \label{5.1} ~

\begin{enumerate}
\item Existence and regularity of solutions to the standard ``ground state equation"
$(- \Delta) u + u - u^p = 0$ ---special cases of which we can recover from {\rm (\ref{ac})} and {\rm (\ref{bo})}---
is discussed carefully in {\rm \cite[Chapter B]{Tao}}.
\item The $m = 0$ case of {\rm (\ref{ac})} is different to the well-known (defocusing) fractional Allen-Cahn equation
\begin{equation} \label{acc}
(- \Delta)^{\gamma/2} u =  u - u^3 \; ,
\end{equation}
already in the classical $\gamma=2$ case:
The classical Allen-Cahn equation is the Euler-Lagrange equation for a positive
definite Lagrangian, see for instance {\rm \cite{WW1,WW2}} (and {\rm \cite[Chapter 5]{BV}} for a Lagrangian formulation
of the full fractional equation {\rm (\ref{acc})}), while Equation {\rm (\ref{ac})} with $\gamma = 2$ arises from the
indefinite Lagrangian functional
$$ \int \frac{1}{2} |\nabla u|^2 + \frac{1}{2} u^2 - \frac{1}{4} u^4 \; . $$
\end{enumerate}
\end{rem}

\smallskip

\smallskip

We finish this paper by observing that $a_0(t)  = |t|^{\gamma/2}$, $0<\gamma<1$, falls outside the class of allowable symbols
of Definition \ref{2.1}, because of the behavior of its derivatives near zero. We will consider this important case elsewhere.
{\em However}, we can still study equations depending on $a_0$ by using the ``$L^2\,$" theory
developed in \cite{G-P-R,G-P-R5} for which only ``ellipticity" of symbols (essentially, assumptions $G_1$ and $G_2$ appearing in Definition
\ref{2.1}) is assumed:

\begin{teo} \label{primero}
We consider $V : \mathbb{R}^n \times \mathbb{R} \rightarrow \mathbb{R}$ such that $V(x,y) = V(|x|,y)$.
Let us assume that $\delta > 0$ is a constant number and take $\gamma > \frac{n}{2}\left( \frac{\delta}{1+\delta} \right)$.
We assume that $\delta$ is small enough so that $0 < \gamma < 1$.
Suppose also that there exist functions $h \in L^2({\mathbb R}^n)$ and $g \in L^{\frac{2(1+\delta)}{\delta}}({\mathbb R}^n)$
such that:
\begin{equation} \label{conditions}
|V(x,y)| \leq C(|h(x)|+|y|^{1+\delta}), \quad
\left| \frac{\partial}{\partial\,y} V(x,y) \right| \leq C\left(|g(x)|+|y|^{\delta}\right)
\end{equation}
for some constant $C > 0\,$. Then, there exist $0 < \epsilon < 1$ and $0 < \rho_\epsilon < 1$
such that there is a radial solution
$u \in \mathcal{H}^{2,2}_r(a)$, where $a(t) = (1/\kappa) t^{\gamma/2}$, to the equation
\begin{equation}
\frac{1}{\kappa}\,(-\Delta)^{\gamma/2} u = - u + V(x,u) \; , \quad \quad \quad \kappa > 0 \; ,
\end{equation}
whenever $\| h \|_{L^2({\mathbb R}^n)} < \rho_\epsilon\,$. Moreover, $u$ belongs to $L^{2(1+\delta)}_r(\mathbb{R}^n)$ and
$\| u \|_{ L^{2(1+\delta)}_r(\mathbb{R}^n)} \leq \epsilon\,$.
\end{teo}
\begin{proof}
We apply Theorem 3.5 from \cite{G-P-R}. Keeping the notation used therein, we set $f(t) = (-t)^{\gamma/2}/\kappa\,$. Then, we see that the function
$t \mapsto f(-t^2) = |t|^\gamma/\kappa$ is non-negative and we also have
\[
\frac{1}{\kappa}\,(1 + |\xi|^2)^{\gamma/2} \leq 2^{\gamma/2}\,\frac{1}{\kappa}\,|\xi|^\gamma = 2^{\gamma/2} f(-|\xi|^2)
\]
for $|\xi| \geq 1$. We conclude that the symbol $f$ belongs to $\mathcal{G}^\gamma$. Now we define $\alpha = 1+\delta$ and we obtain
$\gamma > n(\alpha-1)/(2\alpha)$. Finally, conditions (\ref{conditions}) imply that the function $U(x,y) = - y + V(x,y)$ satisfies
inequalities (21) in \cite{G-P-R}.
\end{proof}

As a first corollary we prove existence of radial solutions to Equation (\ref{pn}): we start with $0 < \gamma < 1$
and fix $\delta > 0$ small enough so that $\gamma > \frac{n}{2}\left( \frac{\delta}{1+\delta} \right)$. We take
$V(x,y) = (1/\kappa)\,d(|x|)\,\sin(y)$. This function $V$ satisfies (\ref{conditions}) for a continuous function $d$ of compact support if we choose
$h=g=d$ and  $\| h \|_{L^2({\mathbb R}^n)} < \rho_\epsilon$. Lastly, we note that we can also prove existence of radial solutions to
Equations (\ref{fls}), (\ref{bo}), and to the $m=0$ case of (\ref{ac}):

\begin{prop}
Let $a(t) = t^{\gamma/2}$. We have:
\begin{itemize}
\item If $1/3 < \gamma < 1$ and $n = 1$ or, if $\,2/3 < \gamma < 1$ and $n = 2$, there exist a radial solution
$u \in \mathcal{H}^{2,2}(a)$ to the equation $(- \Delta)^{\gamma/2} u =  -u + u^3$.

\item If $1/4 < \gamma < 1$ and $n = 1$ or, if $\,1/2 < \gamma < 1$ and $n = 2$ or, if $\,3/4 < \gamma < 1$ and $n = 3$, there exist a radial solution
$u \in \mathcal{H}^{2,2}(a)$ to Equation $(\ref{bo})$.

\item If $\beta/2(\beta+1) < \gamma < 1$ and $n = 1$ or, if $\,\beta/(\beta+1) < \gamma < 1$ and $n = 2$, there exist a radial solution
$u \in \mathcal{H}^{2,2}(a)$ to Equation $(\ref{fls})$.
\end{itemize}
\end{prop}

\paragraph{\bf Acknowledgements:} We thank M. Kowalczyk for information on the Allen-Cahn equation which led us to Remark \ref{5.1},
and the referee for his/her insightful comments.
M.B. has been partially supported by MECESUP
via the project USA0711; H.P. and E.G.R. have been partially
supported by the FONDECYT operating grants \# 1170571  and \# 1161691
respectively.

\bibliographystyle{amsalpha}

\end{document}